\documentclass[11pt]{amsart}
\usepackage{amssymb,amsmath,mathtools}
\usepackage{enumerate}
\usepackage{xcolor}
\usepackage{caption}
\usepackage[centering]{geometry}
\usepackage{graphicx}
\usepackage{color}
\usepackage{amsmath}
\usepackage{amssymb}
\usepackage{amscd}
\usepackage{framed}
\usepackage{bbm}
\usepackage{tcolorbox}

\usepackage{amsthm}

\newcommand{\R}{\mathbb{R}}

\newcommand{\N}{\mathbb{N}}

\newtheorem{theorem}{Theorem}[section]
\newtheorem{proposition}[theorem]{Proposition}
\newtheorem{lemma}[theorem]{Lemma}
\newtheorem{corollary}[theorem]{Corollary}
\theoremstyle{definition}

\newtheorem{remark}[theorem]{Remark}

\newcommand{\E}{\mathbb{E}}
\renewcommand{\P}{\mathbb{P}}
\newcommand{\V}{\mathrm{\mathbb{V}ar}}

\numberwithin{equation}{section}

\newcommand{\eps}{\varepsilon}

\begin{document}

\title{On a variance dependent Dvoretzky-Kiefer-Wolfowitz inequality}

\author{Daniel Bartl}
\address{Faculty of Mathematics, University of Vienna, Austria}
\email{daniel.bartl@univie.ac.at}
\author{Shahar Mendelson}
\address{ The Australian National University, Centre for Mathematics and its Applications}
\email{shahar.mendelson@anu.edu.au}
\keywords{Empirical distribution function, Kolmogorov-Smirnov statistic}
\date{\today}
%\subjclass[2010]{}

\begin{abstract}
	Let $X$ be a real-valued random variable with distribution function $F$. Set $X_1,\dots, X_m$ to be independent copies of $X$ and let $F_m$ be the corresponding empirical distribution function.
	We show that there are absolute constants $c_0$ and  $c_1$ such that if $\Delta \geq c_0\frac{\log\log m}{m}$, then with probability at least $1-2\exp(-c_1\Delta m)$, for every $t\in\R$ that satisfies $F(t)\in[\Delta,1-\Delta]$, 
\[  |F_m(t) - F(t) | \leq  \sqrt{\Delta \min\{F(t),1-F(t)\} } .
\]
	Moreover, this estimate is optimal up to the multiplicative constants $c_0$ and $c_1$.
\end{abstract}

\maketitle
\setcounter{equation}{0}
\setcounter{tocdepth}{1}
%\tableofcontents

\section{Introduction}
Let $X$ be a real-valued random variable and set $X_1,\dots,X_m$ to be independent copies of $X$. Denote by $F(t)=\P(X\leq t)$ the \emph{distribution function} of $X$ and let
\[ F_{m}(t)=\frac{1}{m}\sum_{i=1}^m 1_{(-\infty,t]}(X_i)\]
be the \emph{empirical distribution function}.
A fundamental question  in probability theory and (nonparametric) statistics is whether and to what extent the empirical distribution function $F_m$ can serve as a proxy for the true (unknown) distribution function $F$.
A cornerstone result in that direction is the celebrated \emph{Glivenko-Cantelli} theorem \cite{cantelli1933sulla,glivenko1933sulla}: $\sup_{t\in\R}|F_m(t)-F(t)|$ tends to zero almost surely as the sample size $m$ tends to infinity.
A powerful generalization of the Glivenko-Cantelli theorem is the \emph{Dvoretzky-Kiefer-Wolfowitz (DKW)} inequality  \cite{dvoretzky1956asymptotic} which describes the speed of convergence.
The following formulation of the DKW inequality is taken from  \cite{massart1990tight}.

\begin{theorem}
\label{thm:DKW}
For every $\Delta>0$, with probability at least $1-2\exp(-2\Delta m)$,
\[  \sup_{t\in\R} |F_m(t)  - F(t)  |  \leq  \sqrt\Delta.
\]
%\right)\leq 2\exp(-c\Delta m).
%\end{align}
\end{theorem}

It is well known that Theorem \ref{thm:DKW} is optimal, but at the same time it is reasonable to expect that the deviation  $| F_m(t)  - F(t)  |$ is much smaller than $\sqrt{\Delta}$ when $F(t) $ is significantly smaller or significantly larger than $1/2$.
Specifically, let
\[
\sigma^2(t)
=\V\left( 1_{(-\infty,t]}(X) \right)
= F(t) (1-F(t)),
\]
and observe that
\[
\V\left( F_m(t)  \right)
%=\E\left( \frac{1}{m}\sum_{i=1}^m 1_{(-\infty,t]}(X_i) - F(t) \right)^2
= \frac{ \sigma^2(t) }{m}.
\]
Therefore, a natural conjecture is that the typical behaviour of $| F_m(t) - F(t) |$ is of the order of $\sigma(t)/\sqrt{m}$, and the purpose of this note is show that  such an estimate is (almost) true:

\begin{theorem} \label{thm:ratio}
	There are absolute constants $c_0$ and $c_1$ such that the following holds.
	For every
	\[ \Delta \geq c_0 \frac{ \log \log m }{m} ,\]
	 with probability at least $1-2\exp(-c_1\Delta m)$, for every $t \in \R$ that satisfies $\sigma^2(t)\geq \Delta$ we have
	\[
	\left| F_m(t) - F(t) \right|
	\leq    \sigma(t)  \sqrt{\Delta}.
	\]
\end{theorem}

\begin{remark}
The focus of Theorem \ref{thm:ratio} is on the range $\sigma^2(t)\geq\Delta$ in which  the behaviour of $|F_m(t)-F(t)|$ is governed by cancellations. 
For smaller values of $\sigma^2(t)$ (relative to $\Delta$) no cancellations are to be expected, and if $F(t)\ll \Delta$ then
\[ |F_m(t)-F(t)| \sim \max\left\{F_m(t),F(t)\right\}.\]
The `no cancellations' version of Theorem \ref{thm:ratio} is studied in Section \ref{sec:no.cancellation}.
\end{remark}

At a first glance, the restriction $\Delta\gtrsim \frac{\log\log m}{m}$ in Theorem \ref{thm:ratio} might appear unnatural, especially as it prevents obtaining a uniform error of $\sigma(t)/\sqrt m$.
%Indeed, if $\sigma^2(t)\geq\Delta$, it follows  from Bernstein's inequality that with probability at least  $1-2\exp(-c\Delta m)$, $|F_m(t)-F(t)|\leq\sigma(t)\sqrt\Delta$; and this estimate becomes nontrivial if $c\Delta m > \log(2)$, i.e.\ $\Delta\gtrsim \frac{1}{m}$. 
Note that for a \emph{fixed} $t$ that satisfies $\sigma^2(t)\geq\Delta$, it follows  from Bernstein's inequality that 
\[\P\left(|F_m(t)-F(t)|\geq \sigma(t)\sqrt\Delta \right)\leq 2\exp(-c\Delta m),\] 
which is nontrivial when  $\Delta\gtrsim \frac{1}{m}$. 
On the other hand,  Theorem \ref{thm:ratio} holds \emph{uniformly} in $t$, and then the condition that  $\Delta\gtrsim \frac{\log\log m}{m}$ is necessary:

\begin{theorem}
\label{thm:loglog.optimal}
	Assume that there is a positive sequence $(\Delta_m)_{m=1}^\infty$ and numbers $\alpha,\beta>0$ for which the following holds:
	For every $\Delta\geq \Delta_m$, with probability at least $1-2\exp(-\alpha \Delta m)$, for every $t\in\R$ that satisfies $\sigma^2(t)\geq\Delta$,
	\[ |F_m(t)-F(t)|\leq \beta \sigma(t)\sqrt\Delta.\]
	Then there are constants $c_0$ and $c_1$ that depend only on $\alpha$ and $\beta$ which satisfy that for any  $m\geq c_0$,
	\[\Delta_m \geq c_1\frac{\log\log m}{m}.\]
\end{theorem}

Moreover, not only is the restriction on $\Delta$  in  Theorem \ref{thm:ratio} necessary, the probability estimate is optimal as well and cannot be improved even for a fixed $t\in\R$.

\begin{proposition} \label{lem:lower.bound.probab}
	There are absolute constants $c_1$ and $c_2$ such that the following holds.
	For every $\Delta\geq \frac{c_1}{m}$ and  $t\in\R$ that satisfies $\sigma^2(t)\geq\Delta$, we have that with probability at least $2\exp(-c_1\Delta m)$,
	\[ \left|  F_m(t) - F(t)  \right|
	\geq c_2 \sigma(t) \sqrt{\Delta} .\]
\end{proposition}

\vspace{0.5em}
\noindent
We end the introduction with a word about \emph{notation.}
Throughout, $c,c_0,c_1,\dots$ denote strictly positive absolute constants that may change their value in each appearance.
We write $A\lesssim B$ if $A\leq c B$ for an absolute constant and $A\sim B$ if $c_1 A\leq B\leq c_2 A$ for absolute constants $c_1$ and $c_2$.
If a constant $c$ depends on a parameter $\alpha$, that is denoted by $c=c(\alpha)$.

We denote by $\varepsilon=(\varepsilon_i)_{i=1}^m$ iid Bernoulli random variables (that is, symmetric random variables that satisfy $\P(\varepsilon_1=1)=\frac{1}{2}$) that are independent of $(X_i)_{i=1}^m$.
The expectation only with respect to $(X_i)_{i=1}^m$ (resp.\ $\varepsilon$) is denoted by $\E_X$ (resp.\ $\E_\varepsilon$), and the same convention applies to $\P_X$ and $\P_\varepsilon$.

\section{Proof of Theorem \ref{thm:ratio}}

To simplify the presentation, we assume in what follows that $X$ is absolutely continuous. 
The modifications needed in the general case are straightforward. 
In addition, we may assume without loss of generality that $\Delta \leq \frac{1}{4}$ (because $\sup_{t\in\R}\sigma^2(t)\leq \frac{1}{4})$) and we  consider  only  $t\in\R$ that satisfy $F(t)\leq \frac{1}{2}$; in particular,
$$
\frac{1}{2} F(t)\leq \sigma^2(t)\leq  F(t).
$$
The case where $F(t)>\frac{1}{2}$ follows from the same arguments and is omitted.

The proof of Theorem \ref{thm:ratio} presented here is not the most basic possible, but it is the shortest that we are aware of.
It is based on Talagrand's concentration inequality for empirical process \cite{talagrand1994sharper}.
The version used here can be found, for example, in \cite[Theorem 12.5]{boucheron2013concentration}.

\begin{theorem} 
\label{thm:tal-concentration-emp}
There exists an absolute constant $c$ such that the following holds.
Let $\mathcal{H}$ be a class of functions that are bounded by $1$ and set $\sigma_{\mathcal{H}}=\sup_{h \in \mathcal{H} } \|h\|_{L_2}$.
Then for $\lambda >0$, with probability at least $1-2\exp(-\lambda)$,
\[
\sup_{h \in \mathcal{H} }\left|\frac{1}{m}\sum_{i=1}^m h(X_i) - \E h(X) \right|
\leq c\left( \E \sup_{h \in \mathcal{H}}\left|\frac{1}{m}\sum_{i=1}^m h(X_i) - \E h(X) \right| + \sigma_{\mathcal{H}}\sqrt{\frac{\lambda}{m}} + \frac{\lambda}{m}\right).\]
\end{theorem}

Since $\Delta\leq\frac{1}{4}$, we have that  $\log_2\frac{1}{\Delta} \geq 2$. 
For every $1\leq \ell\leq \log_2\frac{1}{\Delta}$, let  $t_\ell\in\R$ satisfy $F(t_\ell)=2^{-\ell} $ and set $\mathcal{H}_\ell= \{ 1_{(-\infty,t]} : t \in(t_{\ell+1},  t_\ell]\}$. 
In particular, for every $h\in \mathcal{H}_\ell$, $2^{-(\ell+1)}\leq \E h^2(X) \leq 2^{-\ell}$.

\vspace{0.5em}

With Theorem \ref{thm:tal-concentration-emp} in mind, let us begin by showing that for every $1\leq \ell \leq \log_2\frac{1}{\Delta}$,
\begin{equation} 
\label{eq:expectation-F-j}
\mathcal{E}_\ell
=\E \sup_{h \in \mathcal{H}_\ell} \left|\frac{1}{m}\sum_{i=1}^m h(X_i) -\E h(X) \right|
\leq 4 \sqrt{\frac{2^{-\ell}}{m}}.
\end{equation}
Set
\[ \mathcal{V}_\ell= \left\{ (h(X_i))_{i=1}^m : h \in \mathcal{H}_\ell \right\}
\]
and let
\[
\sharp_\ell
= \sup_{h \in \mathcal{H}_\ell} \left| \{i : h(X_i) = 1\} \right|
= \sup_{h \in \mathcal{H}_\ell} \sum_{i=1}^m h(X_i).
\]
Denote by $(e_i)_{i=1}^m$ the standard basis in $\R^m$.
It is straightforward to verify that for every realization of $(X_i)_{i=1}^m$ there is a permutation $\pi$ of $\{1,\dots,m\}$ for which
\begin{equation} \label{eq:structure-1}
\mathcal{V}_\ell
\subset \left\{ \sum_{i=1}^j e_{\pi(i)} : j \leq \sharp_\ell \right\} \cup \{0\}.
\end{equation}
By the Gin\'{e}-Zinn symmetrization inequality \cite{gine1984some} (see also \cite[Lemma 11.4]{boucheron2013concentration}) and \eqref{eq:structure-1},
\begin{align*}
\mathcal{E}_\ell
&\leq  \frac{2}{m}\E_X \E_\eps \sup_{v \in \mathcal{V}_\ell} \left|\sum_{i=1}^m \eps_i v_i \right|
\leq  \frac{2}{m} \E_X \E_\eps \max_{1 \leq j \leq \sharp_\ell} \left|\sum_{i=1}^j \eps_i \right|,
\end{align*}
where  we set $\max_\emptyset=0$.
Invoking L\'{e}vy's maximal inequality (see, e.g., \cite[Proposition 2.3]{ledoux1991probability}),
\[
\E_\eps \max_{1 \leq j \leq \sharp_\ell} \left|\sum_{i=1}^j \eps_i \right|
\leq 2 \E_X \left|\sum_{i=1}^{\sharp_\ell} \eps_i\right|
\leq 2 \sqrt{\sharp_\ell}.
\]
Therefore,
\[
\mathcal{E}_\ell
\leq \frac{4}{m}\E_X \sqrt{\sharp_\ell}
\leq \frac{4}{m}\sqrt{\E \sharp_\ell},
\]
and
\[
\E \sharp_\ell = \E \sup_{ t \leq t_\ell} \sum_{i=1}^m 1_{(-\infty,t]}(X_i)
= \E \sum_{i=1}^m 1_{(-\infty,t_\ell]}(X_i)
\leq m 2^{-\ell},
\]
proving \eqref{eq:expectation-F-j}.

\vspace{0.5em}
Next, let $1\leq \ell \leq \log_2\frac{1}{\Delta}$.
By Talagrand's concentration inequality for $\mathcal{H}_\ell$ and $\lambda =2\Delta m $, it is evident that with probability at least $1-2\exp(-2\Delta m)$,
\begin{align*}
\sup_{h \in  \mathcal{H}_\ell} \left|\frac{1}{m}\sum_{i=1}^m h(X_i) -\E h(X) \right|
&\leq c_1 \left( \sqrt \frac{2^{-\ell}}{ m} + \sqrt{ 2^{-\ell} \Delta} + 2 \Delta\right)  \\
&\leq c_2 \sqrt{ 2^{-\ell} } \sqrt{\Delta},
\end{align*}
where the last inequality holds because $2^{-\ell}\geq \Delta$. 
Moreover, if $h\in \mathcal{H}_\ell$ then \linebreak $\E h^2 (X) \geq  2^{-(\ell+1)}$, and therefore, with probability at least $1-2\exp(-2\Delta m)$,  for every $t\in(t_{\ell+1},t_\ell]$
\begin{align}
\label{eq:ratio}
\left| F_m(t) - F(t) \right|
\leq c_3  \sqrt{ 2F(t) \Delta}
\leq c_4   \sigma(t) \sqrt{\Delta }.
\end{align}

If $\Delta m \geq c \log\log m$ for a suitable constant $c$, then it follows from the union bound over $1\leq \ell\leq \log_2\frac{1}{\Delta}$ that with probability at least $1-2\exp(-\Delta m)$,  \eqref{eq:ratio} holds for every $t\in\R$ that satisfies $F(t)\geq \Delta$.
In particular, \eqref{eq:ratio} holds when $\sigma^2(t) \geq \Delta$, which  completes the proof of Theorem \ref{thm:ratio}.
\qed

\vspace{0.5em}
We end this section by establishing  two simple observations:
That the restriction $\sigma^2(t)\geq\Delta$ in Theorem \ref{thm:ratio} may be omitted by a simple monotonicity argument; and that a high probability error estimate (similar to the one in Theorem \ref{thm:ratio}) trivially implies a corresponding estimate on the expected error, via tail-integration.

\begin{corollary} \label{cor:ratio}
	There are absolute constants $c_0$ and $c_1$ such that the following holds.
	For every $\Delta \geq c_0 \frac{ \log \log m }{m}$,
	 with probability at least $1-2\exp(-c_1\Delta m)$, for every $t \in \R$,
	\[
	\left| F_m(t) - F(t) \right|
	\leq  \Delta +   \sigma(t)  \sqrt{\Delta}.
	\]
\end{corollary}
\begin{proof}
	As before, we will only consider the case $F(t)\leq\frac{1}{2}$.
	By Theorem \ref{thm:ratio}, with probability at least $1-2\exp(-c\Delta m)$,  if $\sigma^2(t)\geq\Delta$ then
	\[|F_m(t)-F(t)|\leq\sigma(t)\sqrt{\Delta}.\]
	Fix a realization of $(X_i)_{i=1}^m$ in that high probability event, and let $t_0\in\R$ satisfy that $\sigma^2(t_0)=\Delta$.
	If $\sigma^2(t)<\Delta$ (and $F(t)\leq\frac{1}{2}$) then $t\leq t_0$, and by the monotonicity of $F_m$ and $F$,
	\[ |F_m(t)-F(t)|
	\leq F_m(t_0)+F(t_0)
	\leq 2 F(t_0) + \sigma(t_0)\sqrt{\Delta} 
	\leq 5 \Delta.
	\qedhere\]
%	and the wanted claim clearly follows.
\end{proof}

\begin{corollary}
\label{cor:tail.integration}
	Let $\Delta_m\geq \frac{1}{m}$ and assume that there are numbers $\alpha,\beta>0$ such that the following holds.
	For every $\Delta\geq \Delta_m$, with probability at least $1-2\exp(-\alpha \Delta m)$, for every $t\in\R$ for which $\sigma^2(t)\geq\Delta$ we have that $|F_m(t)-F(t)|\leq \beta \sigma(t)\sqrt{\Delta}$.
	Then 
	\begin{equation} \label{eq:ratio-expectation}
	\E \sup_{  t\in\R  \, :\, \sigma^2(t)\geq\Delta_m } \left| \frac{F_m(t)-F(t)}{ \sigma(t)  \sqrt{\Delta_m}} \right|
	 \leq c_1
\end{equation}
	where $c_1$ is a constant that depends only on $\alpha$ and $\beta$.
\end{corollary}
\begin{proof}
	Let $\Delta\geq \Delta_m$.
	Just as in the proof of Corollary \ref{cor:ratio}, it follows from the monotonicity of $F$ and $F_m$ that with probability at least $1-2\exp(-\alpha \Delta m)$, for every $t\in\R$,
	\[ |F_m(t)-F(t)|\leq  (4+\beta)\Delta + \beta\sigma(t)\sqrt\Delta. \]
	Moreover, set
	\[ Z=\sup_{t\in\R } \frac{|F_m(t) -F(t)|}{(4+\beta)\Delta_m+\beta \sigma(t) \sqrt{\Delta_m} }\]
	so that 
	\[\sup_{t\in\R \, : \, \sigma^2(t)\geq \Delta_m } \frac{|F_m(t) -F(t)|}{\sigma(t) \sqrt{\Delta_m} }
	\leq c(\beta) Z . \]
%	and set
%	\[ Z=\sup_{t\in\R } \frac{|F_m(t) -F(t)|}{(4+\beta)\Delta_m+\beta \sigma(t) \sqrt{\Delta_m} }.\]
	For every $\lambda\geq 1$ consider $\Delta=\lambda \Delta_m$ and observe that $\P(Z>\lambda) \leq 2\exp(-\alpha\lambda\Delta_m m)$.
	Indeed, with probability at least
	 $1-2\exp(-\alpha \lambda\Delta_m m)$, for every $t\in\R$,
	\begin{align*}
	|F_m(t) -F(t)|
	&\leq (4+\beta) \lambda  \Delta_m+ \beta \sigma(t) \sqrt{\lambda\Delta_m}   \\
	&\leq \lambda \left((4+\beta)\Delta_m + \beta \sigma(t)\sqrt{\Delta_m} \right),
	\end{align*}
	implying that $Z\leq \lambda$.
	Therefore,  by tail-integration,
	\[\E Z \leq 1 + \int_1^\infty \P\left( Z > \lambda\right) \,d \lambda
	\leq 1+  \int_1^\infty 2\exp(-\alpha\lambda \Delta_m m) \,d \lambda
	= 1+\frac{2}{\alpha}.\qedhere\]
\end{proof}

\section{Proof of  Theorem \ref{thm:loglog.optimal}}

The key to the  proof of Theorem \ref{thm:loglog.optimal} is the following result.

\begin{proposition} \label{prop:lower.bound.loglog}
	Assume that there is a number $\gamma>0$ and a positive  sequence $(\Delta_m)_{m=1}^\infty$ such that 
	\begin{equation} \label{eq:ratio-expectation}
	\E \sup_{  t\in\R  \, : \, \sigma^2(t)\geq\Delta_m } \left| \frac{F_m(t)-F(t)}{ \sigma(t)  \sqrt{\Delta_m}} \right|
	 \leq \gamma.
\end{equation}
	Then there are constants $c_0$ and $c_1$ that depend only on $\gamma$ and for any $m\geq c_0$,  $\Delta_m \geq c_1\frac{\log\log m}{m}$.
\end{proposition}

\begin{proof}[{\bf Proof of Theorem \ref{thm:loglog.optimal}} (given Proposition \ref{prop:lower.bound.loglog})]
	By our assumption and Corollary \ref{cor:tail.integration} (applied with $\Delta_m'=\max\{\Delta_m,\frac{1}{m}\}$),
	\[
	\E \sup_{  t\in\R  \, : \, \sigma^2(t)\geq\Delta_m' } \left| \frac{F_m(t)-F(t)}{ \sigma(t)  \sqrt{\Delta_m'}} \right|
	 \leq c_1(\alpha,\beta).\]
	Proposition \ref{prop:lower.bound.loglog} implies that $\Delta_m' \geq c_2\frac{\log\log m}{m}$ for any $m\geq c_3$, where the constants $c_2$ and $c_3$  depend only on $c_1$.
\end{proof}

The rest of this section is devoted to the proof of Proposition \ref{prop:lower.bound.loglog}.

\vspace{0.5em}
We may and do assume without loss of generality that  
\[C_0
\leq \Delta_m m
\leq C_1 \log\log m\]
for well chosen absolute constants $C_0$ and $C_1$.

Set
\[
Z_m = \sup_{t \in \R \, : \, F(t)\in[ 2\Delta_m,\frac{1}{2}] } \left|\frac{ F_m(t)- F(t)}{\sqrt{ F(t)} \sqrt{\Delta_m}} \right| 
\]
so that by our assumption, $\sup_m \E Z_m \leq \gamma$.
We will show that cannot be the case if $\Delta_m \lesssim \frac{\log\log m}{m}$ thanks to a lower estimate on each $\E Z_m$.

\vspace{0.5em}

Our starting point is the following uniform `isomorphic' relation between $F_m$ and $F$.

\begin{lemma} \label{thm:isomorphic}
There exists an absolute constant $c$ such that the following holds.
For every $\Delta \geq \frac{1}{m}$, with probability at least $1-2\exp(-c \Delta m)$,  for every $t\in\R$ that satisfies $\sigma^2(t)\geq\Delta$ we have
\[
\left| F_m(t) -  F(t) \right|
\leq \frac{\sigma^2(t)}{4}.
\]
\end{lemma}
The proof of Lemma \ref{thm:isomorphic} follows the same path as that of Theorem \ref{thm:ratio}. 
It is outlined here for the sake of completeness.
\begin{proof}[Sketch of proof of Lemma \ref{thm:isomorphic}]
	For  $1\leq \ell\leq \log_2\frac{1}{\Delta}$, let $F(t_\ell)=2^{-\ell} $ and set $\mathcal{H}_\ell= \{ 1_{(-\infty,t]} : t \in(t_{\ell+1},  t_\ell]\}$. 
	By Talagrand's concentration inequality and \eqref{eq:expectation-F-j}, for every $\lambda\geq 0$,  with probability at least $1-2\exp(-\lambda)$,
\begin{align*}
\sup_{h \in  \mathcal{H}_\ell} \left|\frac{1}{m}\sum_{i=1}^m h(X_i) -\E h(X) \right|
&\leq c_1 \left( \sqrt \frac{2^{-\ell}}{ m} + \sqrt{ 2^{-\ell}  }\sqrt \frac{\lambda}{m}  + \frac{\lambda}{m} \right) .
\end{align*}
	The proof is completed by setting $\lambda = c_2 2^{-\ell} m  \sim  c_2 F(t_\ell)m$ for a suitable absolute constant $c_2$, followed by the union bound over $\ell$.
\end{proof}

Invoking Lemma \ref{thm:isomorphic}, it is evident that if $C_0$ is a sufficiently large constant and $\Delta_m m \geq C_0$, then there is an event $\Omega$ of probability at least $1/2$ on which
\begin{equation} \label{eq:isomorphic-est}
\frac{3}{4}F(t) \leq F_m(t)\leq \frac{5}{4} F(t)
\end{equation}
for every $t\in \R $ that satisfies $ F(t)\in[ 2\Delta_m,\frac{1}{2}]$.

Let
\[
\mathcal{H}_m = \left\{ \frac{ 1_{(-\infty,t]} }{\sqrt{F(t)}\sqrt{\Delta_m}} : t\in \R \text{ such that } F(t)\in\left[2 \Delta_m,\frac{1}{2}\right]\right\}
\]
and therefore,
\[
Z_m=\sup_{h\in\mathcal{H}_m} \left|\frac{1}{m} \sum_{i=1}^m h(X_i)-\E h(X)\right| .
\]
 By the Gin\'{e}-Zinn symmetrization inequality,
\[
\E Z_m
\geq \frac{1}{2} \E_X \E_\eps \sup_{h\in \mathcal{H}_m} \left| \frac{1}{m} \sum_{i=1}^m \eps_i h(X_i) \right| -  \sup_{h \in \mathcal{H}_m} \frac{|\E h(X)|}{\sqrt{m}}.
\]
Note that 
%$\sup_{h \in \mathcal{H}_m} |\E h(X)| \leq \frac{1}{\sqrt{\Delta_m}}$. 
%Hence,
\[
\sup_{h \in \mathcal{H}_m} \frac{|\E h(X)|}{\sqrt{m}}
=\sup_{t\in\R \, : \, F(t)\in[2\Delta_m,\frac{1}{2}] } \frac{F(t)}{\sqrt{F(t)\Delta_m m}}
\leq \frac{1}{\sqrt{\Delta_m m }}
%= \frac{1}{\sqrt{\alpha_m\log\log m}}
\]
and
\[
\sup_{h \in \mathcal{H}_m} \left| \frac{1}{m} \sum_{i=1}^m \eps_i h(X_i) \right|
= \sup_{ t\in\R \, : \, F(t) \in[2\Delta_m, \frac{1}{2}]  } \frac{1}{ \sqrt{m\Delta_m}} \left| \sum_{i=1}^m \eps_i \frac{ 1_{(-\infty,t]}(X_i) }{\sqrt{ mF(t)}}\right|
=\mathcal{U}_m.
\]
%By \eqref{eq:isomorphic-est} there is an event $\Omega_0$ of probability at least $0.99$ on which for every $t\in R_m$m
%$$
%\frac{1}{2} \PP(X \leq t) \leq \PP_m( X \leq t) \leq \frac{3}{2} \PP(X \leq t).
%$$

Fix $(X_i)_{i=1}^m \in \Omega$ and therefore \eqref{eq:isomorphic-est} is satisfied.
Set $4m\Delta_m \leq \ell \leq \frac{m}{4}$ and let $r_\ell\in\R$ satisfy that $F_m(r_\ell)=\frac{\ell}{m}$.
It follows from \eqref{eq:isomorphic-est} that $F(r_\ell)\in[2\Delta_m,\frac{1}{2}]$ and 
\[ \beta_\ell=\frac{F_m( r_\ell)}{F(r_\ell)} \in\left[\frac{1}{2},2 \right]. \]
%Consider $(e_i)_{i=1}^m$, the standard basis in $\R^m$.
If $\pi$ is the permutation of $\{1,\dots,m\}$ satisfying that $X_{\pi(1)}\leq X_{\pi(2)}$ and so on, then for every $4m\Delta_m \leq \ell \leq \frac{m}{4}$,
\[
\left( \frac{ 1_{(-\infty,r_\ell]}(X_{\pi(i)}) }{\sqrt{mF(r_\ell)}}\right)_{i=1}^m
= \sqrt \frac{\beta_\ell}{\ell} \sum_{i=1}^\ell e_i.
\]

Set
\[
V = \left\{\sqrt \frac{\beta_\ell}{\ell} \sum_{i=1}^\ell e_i : 4m \Delta_m \leq \ell \leq \frac{m}{4} \right\},
\]
and thus
\begin{align*}
\E \,\mathcal{U}_m
&\geq \E_X 1_{\Omega} \frac{1}{\sqrt{ m \Delta_m}} \E_\eps \max_{v \in V } \left|\sum_{i=1}^m \eps_i v_i\right|
\\
&\geq \frac{1}{2\sqrt{m \Delta_m}} \E_\eps \max_{4m \Delta_m \leq \ell \leq \frac{m}{4}}  \sqrt \frac{\beta_\ell}{ \ell} \left| \sum_{i=1}^\ell \eps_i \right|.
\end{align*}

Since $\beta_\ell\geq \frac{1}{2}$,
\begin{align*}
\E \max_{4m \Delta_m \leq \ell\leq \frac{m}{4}}  \sqrt \frac{\beta_\ell} {\ell} \left| \sum_{i=1}^\ell \eps_i \right|
&\geq  \frac{1}{\sqrt 2} \E \max_{4m \Delta_m \leq \ell \leq \frac{m}{4}}  \left| \frac{1}{\sqrt{\ell}}  \sum_{i=1}^\ell \eps_i \right|
\\
&\geq \frac{1}{\sqrt 2 }\E \max_{1 \leq \ell \leq \frac{m}{4}}  \left| \frac{1}{\sqrt{\ell}}  \sum_{i=1}^\ell \eps_i \right| - \frac{1}{ \sqrt 2} \E \max_{1 \leq \ell \leq 4m \Delta_m }  \left| \frac{1}{\sqrt{\ell}}  \sum_{i=1}^\ell \eps_i \right|
\\
&=(1) -(2).
\end{align*}

By H\"{o}ffding's inequality (see, e.g.\, \cite[Lemma 2.2]{boucheron2013concentration}) and the union bound, it is evident that
\[
(2)\leq c_1  \sqrt{\log (m \Delta_m)}.
\]
Moreover, we prove in Theorem \ref{thm:main-lower} that if $m\geq  e^e$, then
\[
(1)\geq c_2\sqrt{\log\log m}.
\]
Since $m\Delta _m \leq C_1\log\log m$,  if $m$ is sufficiently large then
$$
(1)-(2)\geq \frac{c_2}{2}\sqrt{\log\log m}.
$$

Combining these observations shows that
\begin{align*}
\E Z_m
\geq  \E \, \mathcal{U}_m - \frac{1}{\sqrt{\Delta_m m }}
& \geq  \frac{(1) -(2) }{ 4 \sqrt{ m \Delta_m}} - \frac{1}{\sqrt{\Delta_m m }} \\
 &\geq \frac{c_2}{8} \sqrt\frac{ \log\log m }{m \Delta_m} -\frac{1}{\sqrt{\Delta_m m }}.
%& =\frac{1}{ \sqrt{\alpha_m} } \left( \frac{c_2}{8 } - \frac{1}{\sqrt{\log\log m}} \right),
\end{align*}
Finally, as $\Delta_m m\geq 1$ and $\E Z_m\leq \gamma$, it follows that  $\Delta_m m\geq c_3(\gamma) \log\log m$, as claimed.
%showing that if $\liminf_m \alpha_m=0$, then $\limsup_m \E Z_m =\infty$, which is impossible.
\qed

\vspace{0.5em}

Next, let us turn to the proof of the wanted estimate on $(1)$.

\begin{theorem} \label{thm:main-lower}
There is an absolute constant $c$ such that for every $r\geq e^e$,
\[
\E \max_{1 \leq \ell \leq r}  \left| \frac{1}{\sqrt{\ell}}  \sum_{i=1}^\ell \eps_i \right|
\geq c\sqrt{\log \log r}.
\]
\end{theorem}

\begin{remark}
	The \emph{law of the iterated logarithm} implies that, with probability 1,
\[ \limsup_{r\to\infty}  \left|  \frac{1}{\sqrt{ 2r \log\log r}}  \sum_{i=1}^r \eps_i \right| = 1.\]
	Hence, for any $r\geq r_0$,
	\begin{align}
	\label{eq:LIL}
	 \E \max_{1 \leq \ell \leq r}  \left| \frac{1}{\sqrt{2 \ell \log\log \ell } }  \sum_{i=1}^\ell \eps_i \right|
\geq \frac{1}{2}.
	\end{align}
	However, \eqref{eq:LIL}  is a weaker statement than the one in Theorem \ref{thm:main-lower-lemma} (where we have $\log\log r$ instead of $\log\log \ell$), and it  does not suffices for the proof of Proposition \ref{prop:lower.bound.loglog}.
\end{remark}

The proof of Theorem \ref{thm:main-lower} is based on a lower bound on the tail of the sum of iid Bernoulli random variables.

\begin{lemma} \label{lemma:lower-comparison}
There are absolute constants $c_0,c_1,c_2$ such that the following holds.
Let $\eps_1,\dots,\eps_n$ be independent, symmetric $\{-1,1\}$-valued random variables and set $g$ to be the standard gaussian random variable.
Then for $0 \leq \lambda \leq c_0n^{1/6}$,
$$
\P\left(\frac{1}{\sqrt{n}}\sum_{i=1}^n \eps_i \geq \lambda \right)
\geq \P(g \geq \lambda) \cdot \exp\left(c_1 \frac{\lambda^3}{n^{1/2}}\right) \left(1+ c_2 \left(\frac{\lambda+1}{\sqrt{n}}\right)\right).
$$
\end{lemma}

The proof of Lemma \ref{lemma:lower-comparison} can be found, for example, in \cite[Theorem 5.23]{petrov1995limit}.

\vspace{0.5em}

By a tail estimate on the standard gaussian random variable, it is evident that for $\lambda>0$
$$
\P(g \geq \lambda) \geq \frac{1}{\sqrt{2\pi}\lambda} \exp\left(- \frac{\lambda^2}{2}\right) \left(1 - \frac{c}{\lambda^2}\right).
$$
Therefore, there are absolute constants $C_1,C_2,C_3$ such that for any $C_1\leq \lambda \leq C_2n^{1/6}$,
\begin{align}
\label{eq:gaussian-estimate-1}
\P\left(\frac{1}{\sqrt{n}}\sum_{i=1}^n \eps_i \geq \lambda\right)
\geq  \frac{C_3}{\lambda} \exp\left(-\frac{\lambda^2}{2}\right).
\end{align}

\begin{lemma}
\label{thm:main-lower-lemma}
There are integers $\xi, \eta, s_0 \geq 2$  such that the following holds.
Set $m_s = \xi^s$.
Then for any $s \geq s_0$, with probability at least $0.9$,
$$
\max_{s \leq \ell \leq \eta s} \left| \frac{1}{\sqrt{m_\ell \log \log m_\ell}} \sum_{i=1}^{m_\ell} \eps_i \right| \geq \frac{1}{4}.
$$
\end{lemma}
\begin{proof}
	Let
$$
D_\ell = \sum_{i=m_{\ell-1}+1}^{m_\ell}\eps_i
$$ and thus
\[ S_\ell=\sum_{i=1}^{m_\ell}\eps_i
	 = \sum_{i=1}^{m_{\ell-1}}\eps_i  + \sum_{i=m_{\ell-1}+1}^{m_\ell}\eps_i
	 = S_{\ell-1}+D_\ell.
\]
Consider the events
	\begin{align*}
	 \mathcal{A}_\ell
	&= \left\{ D_\ell \geq \frac{1}{2} \sqrt{ m_\ell \log\log m_\ell } \right\}
	\quad\text{and}\\ %\quad
	\mathcal{B}_\ell
	&=\left\{ |S_{\ell-1}| \leq \frac{1}{4} \sqrt{ m_\ell \log\log m_\ell } \right\} ,
	\end{align*}
and note that on the intersection of
	\[
\mathcal{A}=\bigcup_{s\leq \ell \leq \eta s} \mathcal{A}_\ell \quad\text{and}\quad
	\mathcal{B}=\bigcap_{s \leq \ell \leq \eta s} \mathcal{B}_\ell,
\]
there exists $s\leq \ell\leq \eta s$ for which $ \frac{1}{\sqrt{ m_\ell \log\log m_\ell } } S_\ell\geq \frac{1}{4}$. Thus,
$$
\P \left(\max_{s \leq \ell \leq \eta s} \left| \frac{1}{\sqrt{m_\ell \log \log m_\ell}} \sum_{i=1}^{m_\ell} \eps_i \right| \geq \frac{1}{4} \right) \geq \P(\mathcal{A}\cap \mathcal{B}),
$$
and it suffices to establish that $\P(\mathcal{A}\cap \mathcal{B})\geq 0.9$.
	
	\vspace{0.5em}
To show that $\P(\mathcal{A})\geq 0.99$, observe that $\P(\mathcal{A}_\ell)\geq \kappa / \ell$ for a constant $\kappa(\xi)\sim \frac{1}{\log\xi}$.
	Indeed, set
	\[\lambda = \frac{1}{2}\sqrt{ \frac{ m_\ell }{ m_\ell - m_{\ell-1} } \log\log m_\ell }\]
	and therefore
	\[
\P(\mathcal{A}_\ell)
	=\P\left( \frac{1}{ \sqrt{ m_\ell - m_{\ell-1}} } \sum_{i=m_{\ell-1}+1}^{m_\ell} \eps_i  \geq \lambda \right).
\]
	Since $\xi\geq 2$,
	\[\frac{ m_\ell}{m_\ell - m_{\ell-1}} = \frac{\xi}{\xi-1}\in[1,2],
\]
	and in particular $C_1\leq \lambda^2\leq \log\log m_\ell$.
	Moreover, for $\ell\geq s\geq s_0(\xi)$  we have that $\lambda\leq C_2 (m_\ell - m_{\ell-1})^{1/6}$.
	Thus, by \eqref{eq:gaussian-estimate-1},
	\begin{align*}
	\P(\mathcal{A}_\ell)
	&\geq  \frac{C_3}{\lambda}\exp\left(-\frac{\lambda^2}{2} \right) \\
	&\geq \frac{C_3}{ \sqrt{\log\log m_\ell}} \exp\left( -\frac{1}{2} \log\log m_\ell \right)
	\geq \frac{\kappa(\xi)}{\ell}
	\end{align*}
	for $\kappa(\xi)\sim\frac{1}{\log \xi}$, as claimed.

	Now, by the independence of the sets $\mathcal{A}_\ell$,
	\begin{align*}
	\P(\mathcal{A})
	&=1-\prod_{\ell = s }^{\eta s} \P(\mathcal{A}_\ell^c)
	\geq 1-\exp\left( \sum_{\ell = s }^{\eta s} \log\left(1-\frac{\kappa}{\ell} \right) \right) ,
	\end{align*}
	and if $s\geq 5 \kappa$ then $\log(1-\frac{\kappa}{\ell})\geq -2\frac{\kappa}{\ell}$ for every $\ell\geq s$.
	Hence, $\sum_{\ell = s }^{\eta s} \log(1-\frac{\kappa}{\ell} )\geq - c_1\kappa \log \eta$ and 
	\[\P(\mathcal{A})
	\geq 1-\exp(-c_1\kappa\log \eta )
	\geq 0.99,
\]
	where the last inequality holds if $\log\eta \geq c_2(\kappa)$.

	\vspace{0.5em}
	To prove that  $\P(\mathcal{B})\geq 0.99$ one may use H\"{o}ffding's inequality:
\begin{align*}
 \P( \mathcal{B}^c)
&\leq \sum_{\ell = s }^{\eta s} \P(\mathcal{B}_\ell^c)
=  \sum_{\ell = s }^{\eta s} \P\left(  \left|\frac{1}{\sqrt{ m_{\ell-1} } } S_{\ell-1} \right| > \frac{1}{4} \sqrt{\xi \log\log m_\ell } \right)
\\
&\le 2\sum_{\ell = s }^{\eta s} \exp\left( - c_3 \xi \log\log \xi^\ell \right)
\leq 0.01
\end{align*}
provided that $\xi$ is  sufficiently large (and note that the choice of $\xi$ is independent of the choices of $\eta$ and $s$).
\end{proof}

\begin{proof}[\bf Proof of Theorem \ref{thm:main-lower}.]
	Since $\E|\varepsilon_1|=1$, it suffices to prove the claim only for  $r\geq c_0$.
	Using the notation of Lemma \ref{thm:main-lower-lemma}, let $c_0$ be sufficiently large to ensure that the largest integer $s$ for which $ m_{\eta s} \leq r$ satisfies $s \geq s_0$; in particular, $r \leq m_{\eta(s+1)}$. 
	It follows that
	\[
\max_{1\leq \ell \leq r} \left| \frac{1}{\sqrt \ell} \sum_{i=1}^\ell \eps_i \right|
	\geq \max_{s \leq \ell \leq \eta s} \left| \frac{1}{\sqrt{m_\ell }}\sum_{i=1}^{m_\ell} \eps_i \right|
	=(\ast),
	\]
	and by Lemma \ref{thm:main-lower-lemma}, with probability at least $0.9$, $(\ast)\geq \frac{1}{4}\sqrt{\log\log m_s}$.
	
	The final observation is straightforward:
	If $c_0$ (and hence $s$) is sufficiently large, then
	\begin{align*}
	\log \log r
	\leq \log \log m_{\eta (s+1)}
	&= \log \left( \eta (s+1)  \log \xi \right) \\
	&\leq  \log 2\eta + \log \log m_s
	\leq 2 \log\log m_s,
	\end{align*}
and the claim follows.
\end{proof}

\section{Proof of Proposition  \ref{lem:lower.bound.probab}}

The proof of Proposition  \ref{lem:lower.bound.probab} can be established using well-known lower bounds on the binomial distribution. The proof presented here follows a slightly different path and is based on a result due to Montgomery-Smith \cite{montgomery1990distribution}.  
To formulate that result, denote by $\lfloor \lambda \rfloor$ the largest integer smaller than $\lambda$, and for $w \in \R^m$ let $w^\ast$  be the monotone non-increasing rearrangement of $(|w_i|)_{i=1}^m$.

\begin{lemma} 
\label{lemma:MS}
There are absolute constants $c_1,c_2,c_3$ such that for every $0<\lambda < m-1$ and $w\in\R^m$, with probability at least $c_1\exp(-c_2\lambda)$,
	\[
\left| \sum_{i=1}^m \eps_i w_i \right|
	\geq c_3\left( \sum_{i=1}^{ \lfloor \lambda \rfloor } w_i^\ast + \sqrt{\lambda}\left(\sum_{i=\lfloor \lambda\rfloor +1 }^m (w_i^\ast)^2 \right)^{1/2} \right).
\]
\end{lemma}

Again, we only consider the case $F(t)\leq \frac{1}{2}$, and in particular $\frac{1}{2}F(t)\leq \sigma^2(t)\leq F(t)$.

\vspace{0.5em}

Let  $(X_i')_{i=1}^m$ be  independent copies of $X$ that are also independent of $(X_i)_{i=1}^m$ and for $1\leq i \leq m$ set
	\[
W_i=1_{(-\infty,t]}(X_i)-1_{(-\infty,t]}(X_i').
\]
	Note that if $Z$ and $Z'$ are independent random variables with the same distribution and $\lambda\geq 0$, then  $\P(|Z-Z'|\geq 2\lambda)\leq 2 \P(|Z|\geq \lambda)$; therefore,
	\begin{align}
	\label{eq:lower.symm}
	\begin{split}
	\P\left( \left|\frac{1}{m} \sum_{i=1}^m 1_{(-\infty,t]}(X_i)-\E 1_{(-\infty,t]}(X) \right| \geq \lambda\right)
	&\geq \frac{1}{2} \P\left( \left|\frac{1}{m} \sum_{i=1}^m  W_i \right| \geq 2\lambda\right)\\
	&= \frac{1}{2} \P\left( \left|\frac{1}{m} \sum_{i=1}^m \varepsilon_i W_i \right| \geq 2\lambda\right),
	\end{split}
	\end{align}
	where the equality holds by the symmetry and independence of the $W_i$s.
	
	Clearly,
	\[
\P(|W_i|=1)=2F(t)(1-F(t))\geq F(t) \geq \Delta,
\]
and if $\Delta m\geq 8$, then by Markov's inequality, with probability at least $\frac{1}{2}$,
\[
|\{ i : |W_i|=1\}|
\geq \frac{1}{2}F(t)m.
\]
Denote that event by $\Omega$, fix  a realization $(W_i)_{i=1}^m \in\Omega$ and apply Lemma \ref{lemma:MS} for $w=(W_i)_{i=1}^m$ and $\lambda= \frac{1}{4}\Delta m$. Hence,
\[
\P_\eps\left( \left|\frac{1}{m} \sum_{i=1}^m \varepsilon_i W_i \right| \geq \frac{c_1}{4}\sqrt{ F(t) \Delta} \right) \geq c_2\exp\left(-\frac{c_3}{4}\Delta m\right),
\]
and by Fubini's theorem,
	\begin{align*}
	\P\left( \left|\frac{1}{m} \sum_{i=1}^m \varepsilon_i W_i \right| \geq  \frac{c_1}{4} \sqrt{F(t) \Delta} \right)
&\geq \E_{Z} 1_{\Omega} \P_\eps\left( \left|\frac{1}{m} \sum_{i=1}^m \varepsilon_i W_i \right| \geq
\frac{c_1}{4}\sqrt{F(t) \Delta} \right)
\\
	&\geq \frac{c_2}{2}\exp\left(-\frac{c_3}{4}\Delta m\right).
	\hfill
%	 \qed
	\end{align*}
Recalling that $F(t)\sim \sigma^2(t)$, the claim follows from \eqref{eq:lower.symm}.
	\qed

\section{The `no cancellation' range} \label{sec:no.cancellation}

Next, let us turn to the regime in which $\sigma^2(t)$ is small relative to $\Delta$. 
In that range there are no significant cancellations between $\frac{1}{m}\sum_{i=1}^m 1_{(-\infty,t]}(X_i)$ and $F(t)=\E1_{(-\infty,t]}(X)$, and as a result,
\begin{align*}
|F_m(t)-F(t)| \sim %\max\{F_m(t),F(t)\}
\begin{cases}
\max\{F_m(t),F(t)\} &\text{if }F(t)\leq\frac{1}{2} ,\\
\max\{1-F_m(t),1-F(t)\} &\text{if }F(t)>\frac{1}{2}.
\end{cases}
\end{align*}
Specifically, we have the following:

\begin{theorem}
\label{thm:ratio.full.range}
	There are absolute constants $c_0,c_1,c_2$ such that the following holds.
	For every   $\Delta \geq c_0 \frac{ \log \log m }{m} $, 	 with probability at least $1-2\exp(-c_1\Delta m)$, 
	\[
	| F_m(t) - F(t) |
	\leq
	c_2
	\begin{cases}
 	\dfrac{\Delta}{\log(\Delta/\sigma^2(t))}
 		&\text{if } \sigma^2(t) \in \left( \frac{1}{10m} \exp(-\Delta m) , \frac{\Delta}{10}  \right] , \\
 		\\
% 	 \dfrac{1}{m}
% 	 	&\text{if } \sigma^2(t) \in \left( \frac{1}{10m}\exp(-\Delta m) , \frac{\Delta}{10}\exp(-\Delta m)  \right],   \\
 	\sigma^2(t)
 	 	 &\text{if } \sigma^2(t) \leq  \frac{1}{10m} \exp(-\Delta m).
	\end{cases}
	\]
\end{theorem}

\begin{remark}
As we show in what follows, in the range $\left( \frac{1}{10m} \exp(-\Delta m) , \frac{1}{10} \Delta \right]$ the dominant term is $F_m(t)$, while in $(0,\frac{1}{10 m} \exp(-\Delta m)]$ the dominant term is $F(t)$.
\end{remark}

Theorem \ref{thm:ratio.full.range} is complemented by a lower bound that holds for every $t\in\R$ for which $\sigma^2(t)\leq \frac{\Delta}{10}$.
 In particular, the uniform upper estimate in Theorem \ref{thm:ratio.full.range} cannot be improved in a rather strong sense.

\begin{proposition}
\label{prop:lower.full}
	There are absolute constants $c_0,c_1,c_2$ such that the following holds.
	Let $\Delta \geq \frac{c_0}{m}$ and consider $t\in\R$ for which $\sigma^2(t) \leq \frac{\Delta}{10}$.
	Then, with probability at least $2\exp(-c_1\Delta m)$,
	\[
	\left| F_m(t) - F(t) \right|
	\geq  c_2
\begin{cases}
 	\dfrac{\Delta}{\log(\Delta/\sigma^2(t))}
 		&\text{if } \sigma^2(t) \in \left( \frac{1}{10m} \exp(-\Delta m) , \frac{\Delta}{10}  \right] , \\
% 	 \dfrac{1}{m}
 	 \\
% 	 	&\text{if } \sigma^2(t) \in \left( \frac{1}{10m}\exp(-\Delta m) , \frac{\Delta}{10}\exp(-\Delta m)  \right],   \\
 	\sigma^2(t)
 	 	 &\text{if } \sigma^2(t) \leq  \frac{1}{10m} \exp(-\Delta m).
	\end{cases}
	\]
\end{proposition}

\begin{remark}
Just as in the proofs of Theorem \ref{thm:ratio.full.range} and Proposition \ref{prop:lower.full}, we restrict the presentation to absolutely continuous random variables and only consider the case $F(t)\leq \frac{1}{2}$---implying that  $\frac{1}{2}F(t)\leq \sigma^2(t)\leq F(t)$.
In particular,
\[ \left\{ t\in\R: F(t)\leq\frac{1}{2}, \, \sigma^2(t)\in[a,b] \right\}
\subset \left\{ t\in\R: F(t)\leq\frac{1}{2}, \, F(t)\in [a,2a] \right\}.\]

As will become clear in what follows, it is convenient to split the  range $( \frac{1}{10m} \exp(-\Delta m), \frac{\Delta}{10} ]$ to
\[ \left( \frac{1}{10m} \exp(-\Delta m) ,  \frac{\Delta}{10} \exp(-\Delta m)  \right] \cup \left( \frac{\Delta}{10} \exp(-\Delta m)  , \frac{\Delta}{10}   \right] ,\]
noting that if  $\Delta m\geq 10$ then
\begin{align}
\label{eq:error.1.m}
 \frac{\Delta}{\log(\Delta/\sigma^2(t)) }
\sim\frac{1}{m} 
\quad\text{for } \sigma^2(t) \in \left( \frac{1}{10m} \exp(-\Delta m) ,  \frac{\Delta}{10} \exp(-\Delta m)  \right] .
\end{align}
\end{remark}

\subsection{Proof of Theorem \ref{thm:ratio.full.range}}
\hfill

\vspace{0.5em}
\noindent
{\bf $\bullet$ The range 
%$\sigma^2(t) \in \left( \frac{\Delta}{10} \exp(-\Delta m) ), \frac{\Delta}{10} \right]$.
 $F(t) \in \left( \frac{\Delta}{10} \exp(-\Delta m) ), \frac{\Delta}{5} \right]$.}
Since $\frac{\Delta}{F(t)}\geq 5$ and $\log(\lambda)\leq \lambda$ for $\lambda \geq 1$, we have that
$$
F(t) \leq \frac{\Delta}{\log(\Delta /F(t))}.
$$
Hence, it suffices to show that $F_m(t) \leq c \frac{\Delta}{\log(\Delta / F(t)  )}$.

For every $\ell\geq 1$, let $x_\ell\in\R$ satisfy $F(x_\ell)=2^{-\ell}\Delta$.
By Bennett's inequality (see, e.g., \cite[Theorem 2.9]{boucheron2013concentration}),  for every $\ell$, with probability at least $1-2\exp(-3\Delta m)$,
\begin{align}
\label{eq:bennett}
\left| F_m(x_\ell)  -  F(x_\ell) \right|
\leq c_1 \frac{\Delta }{ \log(\Delta / F(x_\ell) )}.
\end{align}
Moreover, if $\Delta m \geq 10$ then by the restriction on $F(t)$,
\[
\left|\left\{ \ell \in\N : F(x_\ell)=\Delta 2^{-\ell}  \geq \frac{\Delta}{10} \exp(-\Delta m) )\right\} \right|
\leq 2 \Delta m.
\]
Invoking the union bound over those integers $\ell$, it follows that with probability at least $1-2\exp(-\Delta m)$ \eqref{eq:bennett} holds uniformly.

Now fix a realization $(X_i)_{i=1}^m$ in that high probability event, set $t\in\R$ in the range in question and let $\ell\in\N$ for which  $t\in(x_{\ell+1},x_\ell]$.
By the monotonicity of $F_m $ and \eqref{eq:bennett},
\[ F_m(t)
\leq F_m( x_\ell)
\leq  (c_1+1) \frac{\Delta }{ \log(\Delta / F(x_\ell) )}
\leq c_2 \frac{\Delta}{ \log(\Delta / F(t))},\]
where the last inequality holds because $F(x_\ell)\leq 2 F( t)$.

\vspace{0.5em}
\noindent
{\bf $\bullet$ The range 
%$\sigma^2(t)\in\left(  \frac{1}{10m} \exp(-\Delta m), \frac{\Delta}{10} \exp(-\Delta m) \right] $.}
%\vskip0.4cm
$F(t)\in\left(  \frac{1}{10m} \exp(-\Delta m), \frac{\Delta}{5} \exp(-\Delta m) \right]$.
}
%Here $F(t)\in\left(  \frac{1}{10m} \exp(-\Delta m), \frac{\Delta}{5} \exp(-\Delta m) \right]$.
Let $t_1 \in\R$ satisfy
$$
p=F( t_1) =\frac{\Delta}{5}\exp(-\Delta m).
$$
If $\lambda\geq 3$, then $\lambda\exp(-\lambda)\leq \exp(-\frac{\lambda}{2})$; hence  if $\Delta m \geq 3$, then 
\begin{align}
\label{eq:bound.p.proof}
pm\leq \frac{1}{5}\exp\left(-\frac{1}{2}\Delta m\right)
\leq \frac{1}{5}.
\end{align}
Denote by $\mathcal{A}_1$ the event in which $|\{ i : X_i\leq t_1\}| \leq  1$.

By \eqref{eq:error.1.m}, it suffices to show that $\P(\mathcal{A}_1)\geq 1-\exp(-c\Delta m)$.
Indeed, on that event,
$$
\sup_{ t \in \R \, : \,  F(t) \leq \frac{\Delta}{5} \exp(-\Delta m) } F_m(t) \leq F_m(t_1) \leq \frac{1}{m},
$$
and  obviously  $F(t) \leq F(t_1) \leq \frac{1}{5m}$.

Clearly,
\begin{align}
\nonumber
\P(\mathcal{A}_1^c)
=\sum_{\ell = 1}^m \binom{m}{\ell} p^\ell (1-p)^{m-\ell}
&\leq \sum_{\ell = 1}^m \exp\left( \ell \log\frac{em}{\ell} + \ell \log p \right)
\\
\label{eq:binom.estimate}
&=\sum_{\ell = 1}^m \exp\left(- \ell \log\frac{\ell}{em p}\right),
\end{align}
and since  $pm\leq \frac{1}{5}\exp(-\frac{1}{2}\Delta m)$, we have that 
\[
\ell \log\frac{\ell}{em p}
\geq \frac{1}{2}\ell\Delta m
\]
for every $\ell \geq 1$.
Comparing  \eqref{eq:binom.estimate} to a geometric progression shows that $\P(\mathcal{A}_1^c)\leq \exp(-c\Delta m)$.

\vspace{0.5em}
\noindent
{\bf $\bullet$ The range 
%$ \sigma^2(t)\leq\frac{1}{10m} \exp(-\Delta m)$.
$F(t)\leq \frac{1}{5m} \exp(-\Delta m)$.}
%\vskip0.4cm
Let $t_2$ satisfy that $p=F(t_2)=\frac{1}{5m} \exp(-\Delta m) $. 
In particular, $t\leq t_2$ for every $t$ in the range in question.
%if $ \sigma^2(t)\leq\frac{1}{10m} \exp(-\Delta m)$ then $t\leq t_2$.
 Denote the event in which $|\{ i : X_i \leq t_2 \}| =0$ by $\mathcal{A}_2$.

Observe that
$\log(1-p)\geq -2p$ (since $p\leq \frac{1}{5}$) and that $\exp(-\lambda)\geq 1-\lambda$ for $\lambda\geq 0$. Hence,
\begin{align}
\label{eq:tiny}
\begin{split}
\P\left(\mathcal{A}_2 \right)
=(1-p)^m
&=\exp(m\log(1-p))\\
&\geq \exp\left( -2mp \right) \\
&\geq 1- 2mp
\geq 1-\exp(-\Delta m).
\end{split}
\end{align}
For every realization $(X_i)_{i=1}^m\in\mathcal{A}_2$ and $t\leq t_2$, we have $F_m(t)\leq F_m(t_2)=0$. 
Therefore,
\[  |F_m(t) - F(t) |
=F(t) \sim \sigma^2(t),  \]
which concludes the proof of Theorem \ref{thm:ratio.full.range}.
\qed

\vskip0.5cm
\subsection{ Proof of Proposition \ref{prop:lower.full}.}
\hfill
%As before, we only consider $t\in\R$ for which $F(t)\leq\frac{1}{2}$; hence $\frac{1}{2} F(t)\leq\sigma^2(t)\leq F(t)$.

\vspace{0.5em}
\noindent
\vspace{0.5em}
\noindent
{\bf $\bullet$ The range 
%$\sigma^2(t) \in \left( \frac{\Delta}{10} \exp(-\Delta m) ), \frac{\Delta}{10} \right]$.
 $F(t) \in \left( \frac{\Delta}{10} \exp(-\Delta m) ), \frac{\Delta}{5} \right]$.}
%\hfill
%
%\vspace{0.5em}
%\noindent
Fix $t$ in that range and set $p=F(t)$.
% \in \left(\frac{\Delta}{10}\exp(-\Delta m), \frac{\Delta}{5} \right] $. 
Since $\frac{\Delta}{p}\geq 1$ and $\log \lambda\leq \lambda$ for $\lambda\geq 1$, we have that
$$
p \leq \frac{ \Delta  }{ \log(\Delta / p)}.
$$
Moreover, set
\[
k= 4 \frac{ \Delta m }{ \log(\Delta / p)}
\]
and note that by the restriction on $p$,  $k\geq \max\{1,4mp\}$.

Let $\mathcal{B}_1$ be the event in which $|\{ i : X_i\leq t\} | \geq k$. It is straightforward to verify that
\begin{align*}
\P(\mathcal{B}_1)
&\geq \binom{m}{k} p^k (1-p)^{m-k} \\
&\geq \exp\left( k \log \frac{m}{k} + k \log p + (m-k)\log(1-p) \right)
\\
&\geq \exp\left( - k \log \frac{k}{mp}  - 2mp\right),
\end{align*}
where we used that $\log(1-p)\geq -2p$ and that $\frac{k}{mp} \geq 1$.

By the definition of $k$,
\[
 k \log\frac{k}{mp}
=  4\frac{ \Delta m }{ \log(\Delta / p)} \log \frac{ 4 \Delta  }{p \log(\Delta / p)}
\leq  8\Delta m,
\]
and recalling that $2 pm \leq \Delta m$ we have $\P(\mathcal{B}_1)\geq \exp(-9\Delta m)$.

Finally, for a realization $(X_i)_{i=1}^m\in\mathcal{B}_1$,
\begin{align*}
 F_m(t) -F(t)
\geq \frac{k}{m} - p
&= 4 \frac{ \Delta  }{ \log(\Delta / p)} - p
\\
&\geq 3\frac{ \Delta  }{ \log(\Delta / p)}
\sim\frac{ \Delta  }{ \log(\Delta / \sigma^2(t))},
\end{align*}
as claimed.

\vspace{0.5em}
\noindent
{\bf $\bullet$ The range 
%$\sigma^2(t)\in\left(  \frac{1}{10m} \exp(-\Delta m), \frac{\Delta}{10} \exp(-\Delta m) \right] $.}
%\vskip0.4cm
$F(t)\in\left(  \frac{1}{10m} \exp(-\Delta m), \frac{\Delta}{5} \exp(-\Delta m) \right]$.
}
%\vskip0.4cm
%\hfill
%
%\vspace{0.5em}
%\noindent
Fix $t$ in that range and set $p=F(t)$.
%\in \left[ \frac{1}{10m}\exp(-\Delta m), \frac{\Delta}{5}\exp(-\Delta m) \right) $.
Let $\mathcal{B}_2$ be the event in which $|\{ i : X_i\leq t\} | \geq 1$, i.e.,
\begin{align*}
\P(\mathcal{B}_2)
=mp
\geq \frac{1}{10}\exp(-\Delta m).
\end{align*}

On the other hand, it follows from the restriction on $p$ that  if $\Delta m\geq 3$ then $F(t)=p\leq \frac{1}{5m}$ (see \eqref{eq:bound.p.proof}).
Therefore,  in $\mathcal{B}_2$,
\[ |F_m(t) - F(t) |
\geq \frac{1}{m}-F(t)
\geq \frac{4}{5m }.
 \]
By \eqref{eq:error.1.m}, the wanted estimate follows.

\vspace{0.5em}
\noindent
{\bf $\bullet$ The range 
%$ \sigma^2(t)\leq\frac{1}{10m} \exp(-\Delta m)$.
$F(t)\leq \frac{1}{5m} \exp(-\Delta m)$.}
%\vskip0.4cm
%
Fix $t$ in that range.
%and set $ in particular $F(t)\leq \frac{1}{5m}\exp(-\Delta m)$.
It was shown in \eqref{eq:tiny} that with probability at least $1-\exp(-\Delta m)$, $|\{ i : X_i\leq t\}|=0$.
In that event, $F_m( t)=0$ and
\[  |F_m(t) - F(t)|
=F(t)
\geq \sigma^2(t) \]
as claimed.
This concludes the proof of Proposition \ref{prop:lower.full}.
\qed

\vspace{1em}
\noindent
\textsc{Acknowledgements:}
The first author is grateful for financial support through the Austrian Science Fund (FWF) projects ESP-31N and P34743N.

\bibliographystyle{plain}
%\bibliography{../ref_HDP.bib}

\end{document}